\documentclass[reqno]{amsart}
\usepackage{mathrsfs}
\usepackage{amsfonts}
\usepackage{amsmath}

\begin{document}

\begin{center}
{\bf\LARGE The Second Main Theorem Concerning Small Algebroid
Functions.}$^*$
\end{center}
 \vskip.1in
  \vskip.1in

\begin{center} Daochun Sun\\
{ \footnotesize ( School of Mathematics, South China Normal
University, Guangzhou 510631, China)}\\
 Zongsheng Gao\\
  {\footnotesize ( LMIB and Department of Mathematics, Beijing
University of Aeronautics and Astronautics, Beijing 100083,
China)}\\
Huifang Liu\\
{ \footnotesize ( School of Mathematics, South China Normal
University, Guangzhou 510631, China)}\\
\end{center}
 \vskip.1in
{\bf  Abstract.}
 In this paper, we firstly give the definition of meromorphic function element and algebroid mapping.
 We also construct the
algebroid function family in which the arithmetic, differential
operations is closed. On basis of these works, we firstly proved the
Second Main Theorem
concerning small algebroid functions for $v$-valued algebroid functions.\\
{\bf Keywords.} algebroid function, algebroid mapping, corresponding
addition,
the Second Main Theorem. \\
{\bf MSC(2000).} 32C20,  30D45.\\
\footnote[0] {\\
$^*$This work is supported by the National Nature Science Foundation
of China(No.10771011, 10871076).}

\numberwithin{equation}{section}
\newtheorem{theorem}{Theorem}[section]
\newtheorem{lemma}[theorem]{Lemma}
\newtheorem{conjecture}{Conjecture}
\newtheorem{corollary}{Corollary}[section]
\newtheorem{question}{Question}
\newtheorem{remark}{Remark}[section]
\newtheorem{definition}{Definition}[section]
\allowdisplaybreaks \large

\section{Introduction}
In 1925, R. Nevanlinna obtained the Second Main Theorem for
meromorphic functions and posed the problem whether the the Second
Main Theorem can be extended to small functions (See \cite{a1}.).
Dealing with the problem, Q. T. Chuang proved the Second Main
Theorem still holds for small entire functions (See \cite{a2},
\cite{a3}.). Until 1986, the problem was solved by N. Steinmetz (See
\cite{a4}.). In 2000, M. Ru proved the Second Main Theorem
concerning small meromorphic functions for algebroid functions (See
\cite{a5}.).

It is natural to consider the problem whether the Second Main
Theorem for algebroid functions is still true when we replace the
small meromorphic functions by small algebroid functions. Before
considering  the problem, we must define the arithmetic,
differential operations over algebroid functions. Hence we give the
definition of meromorphic function element, algebroid mapping and
construct the algebroid function family $H_W$. In $H_W$ the
arithmetic, differential operations is closed. On basis of these
works, by using the method of Reference \cite{a6}, we proved the
Second Main Theorem concerning small algebroid functions.

 Suppose that $A_v(z),\cdots,A_0(z)$ are
analytic functions without common zeros in the complex plane $C$.
Then the binary complex equation
$$\Psi(z,W)= A_v(z)W^v+A_{v-1}(z)W^{v-1}+\cdots+A_0(z)=0 $$
defines a $v$-valued algebroid function $W(z)$ in the complex plane
$C$. The above equation can be transformed to the standard equation
$$\Psi^*(z,W)= W^v+A^*_{v-1}(z)W^{v-1}+\cdots+A^*_0(z)=0,$$ where
$A^*_t(z):=\frac{A_t(z)}{A_v(z)}~~ (t=0,1,2,\cdots,v-1)$ are
meromorphic functions in the complex plane $C$. Note that for a
$v$-valued algebroid function $W(z)$,  its standard equation is
unique.

If $\Psi(z,W)$ is irreducible, then the corresponding $W(z)$ is
called a $v$-valued irreducible algebroid function. For an
irreducible algebroid function $W(z)$, the points in the complex
plane can be divided to two classes. One is a set $T_W$ of regular
points of $W(z)$, the other is a set $S_W=C-T_W$ of critical points
of $W(z)$. The set $S_W$ is an isolated set (See \cite{a7},
\cite{a8}.).

In this paper, $\Psi(z,W)$ needn't be irreducible in the usual case.
A $v$-valued algebroid function $W(z)$ may decompose to $n(\geq 1)$
number of $v_n$-valued irreducible algebroid functions(containing
the case $W$ is a complex constant) and $v=\sum^v_{j=1}v_j$.

For a $v$-valued reducible algebroid function $W(z)$, its
corresponding binary complex equation $\Psi(z,W)=0$ can be
decomposed to the product of $q(\leq v)$ non-meromorphic coprime
factors, namely
$$\Psi(z,W)=\Psi_1(z,W)\Psi_2(z,W)\cdots\Psi_q(z,W)=0.$$
Let $S_j$ denote the set of critical points of the irreducible
complex equation $\Psi_t(z,W)=0$. We define the set of critical
points of reducible algebroid function $W(z)$ by
$S_W:=\cup^q_{j=1}S_j$ (Since $\{S_j\}(j=1, \cdots, q) $ are all
isolated sets, $S_W$ is also an isolated set.), the set of regular
points of reducible algebroid function $W(z)$ by $T_W:=C-S_W$.

\begin{remark}
If $q=1$, then $W(z)$ is an irreducible algebroid function.
\end{remark}

\begin{remark}
If $(q(z),b)$ is a polar element or a multivalent
algebraic function element, then $b\in S_W$.
\end{remark}

\begin{remark}
For every $a\in T_W$, there exist and only exist $v$ number of
regular function elements $\{(w_t(z),a)\}^v_{t=1}$. In this paper,
we usually denote $W(z)=\{w_j(z)\}^v_{j=1}$. If there exists $1\leq
t<j\leq v$ such that $w_t(z)\equiv w_j(z)$, then the complex
equation $\psi(z,W)=0$ must have non-meromorphic function multiple
factor.
\end{remark}
In this paper, we use the standard notations of the value
distribution for algebroid functions (See \cite{a7}.).

\section{Some basic properties of algebroid functions}

\begin{definition}
Let $W(z)$ and $M(z)$ be two  algebroid functions defined by
$$\Psi(z,W)= A_v(z)W^v+A_{v-1}(z)W^{v-1}+\cdots+A_0(z)= A_v(z)\prod^v_{j=1}(W-w_j(z))=0,
~~A_v(z)\not\equiv 0\eqno(2.1)$$
 and
$$\Phi(z,M)= B_s(z)M^s+B_{s-1}(z)M^{s-1}+\cdots+B_0(z)= B_s(z)\prod^s_{t=1}(M-m_t(z))=0,
~~B_s(z)\not\equiv 0,\eqno(2.2)$$ respectively, $W(z)$ and $M(z)$
are called identical, write $W(z) \equiv M(z)$, provided that $v=s$
and the corresponding coefficients are proportional, namely
$$
E(z):=\frac{A_v(z)}{B_v(z)}=\frac{A_{v-1}(z)}{B_{v-1}(z)}=\cdots=\frac{A_0(z)}{B_0(z)}.$$
\end{definition}
Since the coefficients of the equations (2.1) and (2.2) haven't
common zeros, $E(z)$ is a nonzero constant or an analytic function
without zeros.

\begin{theorem}
Suppose that $W(z)=\{w_j(z)\}^v_{j=1}$ and $M(z)=\{m_t(z)\}^s_{t=1}$
are two irreducible algebroid functions defined by (2.1) and (2.2),
respectively. The following conditions are equivalent:

(1) $W(z) \equiv M(z)$.

(2) There exist some regular function elements $(w_j(z),b)$ of
$W(z)$ and $(m_j(z),b)$ of $M(z)$ such that $(w_j(z),b)=(m_j(z),b)$.

(3) The eliminant $R(\Psi,\Phi) \equiv 0$.\\
\end{theorem}

\begin{proof}
(1)$\Rightarrow$(3):
$$R(\Psi,\Phi)=A^s_v(z)\prod^v_{j=1}\Phi(z,w_j(z))
=E(z)A^s_v(z)\prod^v_{j=1}\Psi(z,w_j(z))\equiv 0.$$ By the property
of the eliminant, the first equal sign holds (See \cite{a9}.). Then
by Definition 2.1, we get the second equal sign. Since $(w_j(z),
z)(j=1, \cdots, v)$are regular function elements belong to (2.1),
$\Psi(z,w_j(z))\equiv 0$ in some neighborhood of $z$. Combining the
identical principle of analytic functions, we get the third equal
sign.

(3)$\Rightarrow$(2):Since
$$R(\Psi,\Phi)=A^s_v(z)\prod^v_{j=1}\Phi(z,w_j(z))=A^s_v(z)B^v_s(z)\prod^v_{j=1}\prod^s_{t=1}
(w_j(z)-m_t(z))\equiv 0.$$ there at least exists some term
$w_j(z)-m_t(z)\equiv 0$. Hence  there exist some regular function
element $(w_j(z),a)$ of $W(z)$ and $(m_j(z),a)$ of $M(z)$ such that
$(w_j(z),a)=(m_j(z),a)$.

(2)$\Rightarrow$(1): Since the irreducible algebroid function is a
connected Riemann surface, the two identical regular function
elements can be continued analytically to their Riemann surface
respectively,  such that the corresponding regular function elements
are all identical. Hence $v=s$. Then combining the Viete theorem, we
get
$$\frac{A_{t}(z)}{A_v(z)}=\frac{B_{t}(z)}{B_s(z)}=
\sum (-1)^{v-t}w_{n_1}(z)w_{n_2}(z)\cdots w_{n_{v-t}}(z)
(t=0,1,2,\cdots,v-1),$$ where $w_{n_1}(z), w_{n_2}(z), \cdots,
w_{n_{v-t}}(z)$ denote any given $v-t$ distinct elements among
$w_1(z), \cdots, w_v(z)$. From this we can obtain (1).
\end{proof}

Note that by Theorem 2.1, an irreducible algebroid function $W(z)$
can not contain two same regular function elements.

\begin{theorem}
Suppose that $W(z)=\{(w_{j}(z),B(a,r_a))\}^v_{j=1}$ is a $v$-valued
algebroid function defined by (2.1). If it contains two same regular
function elements, then there exist two same $m $-valued ($2m\leq
v$) algebroid functions decomposed from $W(z)$. Hence $W(z)$ is
reducible.
\end{theorem}

\begin{proof}
Suppose that $(w_j(z),B(a,r_a))\equiv(w_t(z),B(a,r_a))$. Then
$$R(\Psi, \Psi_W) =(-1)^{\frac{v(v-1)}{2}}A^{2v-1}_v(z)\prod_{1\leq j< t \leq v}(w_j(z)-w_t(z))^2
\equiv 0.$$ By Theorem 2.4 in reference [7],  $\Psi(z,W)$ must have
the non-meromorphic function multiple factor. Hence there exist two
same $m$-valued ($2m\leq v$) algebroid function decomposed from
$W(z)$. So $W(z)$ is reducible.
\end{proof}

\begin{definition}
Meromorphic function element is defined by $(q(z),B(a,r))$, where
$q(z)$ is analytic in the disc $B_0(a,r):=\{0<|z-a|<r\}$ and $a$ is
not a essential point. So $q(z)$ can be expressed by Laurent series
$q(z)=\sum^{\infty}_{n=t}a_n(z-a)^n ~(a_t\ne 0)$. We also denote it
by $(q(z),a)$. If the above $t<0$, then we call $(q(z),a)$ is a
truth meromorphic function element. Especially if $q(z)\equiv c$
($c$ denotes a constant.).\\
Two meromorphic function elements $(q(z),a)$ and $(p(z),b)$ are
called identical provided that $a=b$ and there exists $r>0$ such
that $q(z)\equiv p(z)$ in the disc $B_0(a,r)$.\\
If $\Psi(z,q(z))=0$ holds for any $z\in B_0(a,r)$, then $(q(z),a)$
is called a meromorphic function element of algebroid function
$W(z)$ or $\Psi(z,W)=0$.
\end{definition}

\begin{remark}The regular function element is also the meromorphic function
element.
\end{remark}

\begin{definition}
The regular function element $(p(z),B(b,R_b))$ is called the direct
continuation of meromorphic function element $(q(z),B(a,R_a))$
provided that $b\in B(a,R_a)$ and in the domain $B(a,R_a)\cap
B(b,R_b)$ we have $p(z)\equiv q(z)$.\\
For any $\epsilon\in(0,R_a)$, the set of meromorphic function
element $(q(z),B(a,R_a))$ and all direct continuation of meromorphic
function element $(q(z),B(a,R_a))$ in the disc $B_0(a,\epsilon)$ is
called a neighborhood of $(q(z),B(a,R_a))$. We denote it by
$V_\epsilon(q(z),a)$.
\end{definition}

\begin{remark}
For any given point in $B_0(a,R_a)$, the direct continuation is
uniqueness.
\end{remark}

\begin{remark}
The direct continuation of meromorphic function element must be
regular function element. Hence the truth meromorphic function
element is isolated.
\end{remark}

\begin{definition}
Let $W(z)=\{(w_{a,j}(z),a)\}$ be a $v$-valued algebroid function.
$h$ is called an algebroid mapping of $W(z)$ if $h$ satisfies the following conditions.\\
(i)Uniqueness: For any regular function element $(w_{a,j}(z),a)$,
its image element $h\circ (w_{a,j}(z),a)= (h\circ w_{a,j}(z),a)$ is meromorphic function element and unique.\\
(ii)Continuation: For any image element $(h\circ w_{a,j}(z),a)$,
there exists  $\epsilon=\epsilon(h\circ w_{a,j}(z),a)>0$ such that
for any regular function element $(w_b(z),b)\in
V_{\epsilon}(w_{a,j},a)$, we have $(h\circ w_b(z),b)\in
V_{\epsilon}(h\circ w_{a,j},a)$.\\
(iii)Weak boundary: If $a\in S_W$, then $h$ is weak bounded at the
neighborhood of $a$. Namely there exist integer $p>0$, real numbers
$r>0$ and $M>0$, such that for any $b\in B_0(a,r):=\{z; 0<|z-a|<r\}
\subset T_W$ and any $t=1,2,\cdots,v$, the corresponding image
element $(h\circ w_{b,t}(z),b)$ are all the regular function
elements and satisfies $|(b-a)^ph\circ w_{b,t}(b)|<M$.
\end{definition}

\begin{theorem}
Let $h$ be an algebroid mapping of $v$-valued algebroid function
$W(z)=\{(w_{a,j}(z),a)\}$. Then\\
(1)$h\circ W(z):=\{(h\circ w_{a,j}(z),a)\}$ is a $v$-valued
algebroid function. \\
(2)If $W(z)$ is irreducible, then $h\circ
W(z)$ is irreducible if and only if $h$ is injective. Namely $h\circ
(w(z),a)\ne h\circ (m(z),b)$) when $(w(z),a)\ne (m(z),b)$, where
$(w(z),a)$ and $(m(z),b)$ are regular function elements.
\end{theorem}

\begin{proof}
For any $z_0\in T_W$, if there exists some truth meromorphic
function element among the corresponding meromorphic image elements
$\{(h\circ w_{z_0,j}(z),z_0)\}^v_{j=1}$, then $z_0$ is called a pole
of $h$. We denote by $P_h$ the set of poles of $h$. By the
continuation of $h$, we know that $P_h$ is an isolated set.

(1)Firstly we define the analytic functions
$\{H^*_t(z)\}^{v-1}_{t=0}$ in $T_W-P_h$. For any $z_0\in T_W - P_h$,
the corresponding image elements $\{(h\circ
w_{z_0,j}(z),z_0)\}^v_{j=1}$ are all regular function elements. Set
$$H^*_t(z_0)=\sum (-1)^{v-t}[h\circ w_{z_0,j_1}(z_0)]\cdot[h\circ
w_{z_0,j_2}(z_0)]\cdot ...\cdot [h\circ
w_{z_0,j_{v-t}}(z_0)],\hspace{0.3cm} t=0,1,2,...,v-1. $$ By the
continuation of $h$, there exists $\epsilon$, such that for any
$y\in B(z_0,\epsilon)$, the corresponding image elements $\{(h\circ
w_{y,j}(z),y)\}$ are the  direct continuation of $\{(h\circ
w_{z_0,j}(z),z_0)\}$ respectively. Namely we have $h\circ
w_{y,j}(z)\equiv h\circ w_{z_0,j}(z)$ in the neighborhood of $y$. So
we have $$ H^*_t(y)=\sum (-1)^{v-t}[h\circ w_{y,j_1}(y)]\cdot[h\circ
w_{y,j_2}(y)]\cdot ...\cdot [h\circ w_{y,j_{v-t}}(y)]$$
$$=\sum (-1)^{v-t}[h\circ w_{z_0,j_1}(y)]\cdot[h\circ w_{z_0,j_2}(y)]\cdot ...\cdot [h\circ w_{z_0,j_{v-t}}(y)]. $$
Hence in $B(z_0,\epsilon)$, for any $t=0,1,...,v-1$ we have
$$
H^*_t(z)\equiv \sum (-1)^{v-t}[h\circ w_{z_0,j_1}(z)]\cdot[h\circ
w_{z_0,j_2}(z)]\cdot ...\cdot [h\circ w_{z_0,j_{v-t}}(z)]. $$ So
$\{H^*_t(z)\}$ is analytic in $B(z_0,\epsilon)$. By Viete theorem,
they define the following complex equation
$$\Phi^*(z,W)= W^v+H^*_{v-1}(z)W^{v-1}+...+H^*_0(z)=\prod^v_{j=1}[W-h\circ w_{z_0,j}(z)]=0$$
and $\Phi^*(z,h\circ w_{z_0,j}(z))=0$ in $B(z_0,\epsilon)$. Since
$z_0$ is arbitrary, $\{H^*_t(z)\}^{v-1}_{t=0}$ are analytic in
$T_W-P_h$.

When $z_0\in S_W\cup P_h$, since $h$ is weak bounded, $z_0$ is the
isolated singular point and is not the essential isolated singular
point of $\{H^*_t(z)\}$. This shows that $\{H^*_t(z)\}^{v-1}_{t=0}$
are meromorphic in the complex plane and the corresponding complex
equation $\Phi^*(z,W)=0$ defines the algebroid function $h\circ
W(z)$.

(2)Suppose that $h$ is injective. For any two regular image elements
$(h\circ w_{a,j}(z),a)\ne(h\circ w_{b,t}(z),b)$, they define
uniquely two distinct regular primary image elements $(h\circ
w_{a,j}(z),a)\ne(h\circ w_{b,t}(z),b)$. Take a path $\gamma\subset
T_W\cap T_{h\circ W}$ such that two  primary image elements can be
continued analytically each other along $\gamma$. By the
continuation of $h$, we know that $(h\circ w_{a,j}(z),a)$ and
$(h\circ w_{b,t}(z),b)$ can be connected by
$\gamma$. Hence $h\circ W(z)$ is irreducible.\\
Conversely suppose that there exist two different regular function
elements $(w_{a,j}(z),a)\ne (w_{a,t}(z),a)$($j\ne t$) such that the
corresponding image elements $(h\circ w_{j}(z),a)= (h\circ
w_{t}(z),a)$). Then by Theorem 2.2, $h\circ W(z)$ is reducible.
\end{proof}

\begin{definition}
Suppose that $W(z)=\{(w_{j}(z),a)\}$ is a $v$-valued algebroid
function defined by the following complex equation
$$\Psi(z,w)=A_v(z)W^v+A_{v-1}(z)W^{v-1}+...+A_1(z)W+A_0(z)$$
$$=A_v(z)(W-w_1(z))(W-w_2(z))...(W-w_v(z))=0,$$
and $f(z)$ is meromorphic in the complex plane $C$.

1) Define $h_{-W}\circ (w_{j}(z),a):=(-w_{j}(z),a)$. By Viete
theorem, the complex equation with respect to $h_{-W}\circ W(z)$ is
$$\Psi_{-W}(z,w):=A_v(z)(W-(-w_1(z)))(W-(-w_2(z)))...(W-(-w_v(z)))$$
$$=A_v(z)W^v-A_{v-1}(z)W^{v-1}+...+(-1)^vA_0(z)=0.$$
The $v$-valued algebroid function $h_{-W}\circ W(z)$ is called the
negative element of $W(z)$. We denote it by $-W(z)$, denote the
algebroid mapping $h_{-W}$ by $-h$.

2) Define $h_{1/W}\circ (w_{j}(z),a):=(\frac{1}{w_{j}(z)},a)$.By
Viete theorem, the complex equation with respect to $h_{1/W}\circ
W(z)$ is
$$\Psi_{1/W}(z,w):=A_v(z)(W-\frac{1}{w_1(z)})(W-\frac{1}{w_2(z)})...(W-\frac{1}{w_v(z)})$$
$$=A_0(z)W^v-A_1(z)W^{v-1}+...+A_v(z)=0.$$
The $v$-valued algebroid function $h_{1/W}\circ W(z)$ is called the
inverse element of $W(z)$. We denote it by $\frac{1}{W(z)}$, denote
the algebroid mapping $h_{1/W}$ by $\frac{1}{h}$.

\begin{remark}
Especially, $W(z)\equiv 0$ is also the algebroid function. Its
inverse element is defined as $\frac{1}{W(z)}\equiv \infty$ and
$\frac{1}{W(z)}$ is also the algebroid function.
\end{remark}

3) Define $h_f\circ (w_{j},a)=(f(z),a)$. It is easy to prove that
$h_f$ satisfies Definition 2.4. So $h_f$ is an algebroid mapping. By
Theorem 2.3, The $v$-valued algebroid function $h_f \circ
W(z)=\{f(z)\}$ are $v$ same meromorphic functions $f(z)$.
Especially, if $f(z)\equiv c\in {\overline C}$, then the algebroid
function $h_c\circ W(z)=\{c\}$ degenerates into $v$ same finite or
infinite complex constants.

4) Define $h_{W^\prime}\circ (w_j(z),a)=(w^\prime_j(z),a)$. It is
easy to prove that $h_{W^\prime}$ satisfies the conditions 1, 2 of
Definition 2.4. If $z_0\in S_W$, then in
$B_0(z_0,r):=\{0<|z-z_0|<r\}$ we have $$
q_t(z):=\sum^\infty_{n=u_t}a_{n,t}(z-a_0)^{n/\lambda_t},~t=1,2,...,m,$$
where $\lambda_t$ is a positive integer, $u_t$ is an integer and
$\sum^m_{t=1}\lambda_t=v$. It is easy to see that
$$h_{W^\prime}\circ q_t(z)=\sum^\infty_{n=u_t}\frac{na_{n,t}}{\lambda_t}(z-a_0)^{\frac{n-\lambda_t}
{\lambda_t}},~t=1,2,...,m$$ is weak bounded. By Theorem 2.3,
$h_{W^\prime}\circ W(z)$ defines a $v$-valued algebroid function. We
call it the derivative of $W(z)$. We denote it by $h_{W^\prime}\circ
W(z)=W^\prime(z)$. The complex equation with respect to
$W^\prime(z)$ is
$$\Psi^\prime(z,w):=B_v(z)(W^\prime-w^\prime_1(z))(W^\prime-w^\prime_2(z))...(W^\prime-w^\prime_v(z))$$
$$:=B_v(z)(W^\prime)^v+B_{v-1}(z)(W^\prime)^{v-1}+...+B_1(z)W^\prime+B_0(z)=0.$$
\end{definition}

\begin{definition}
Let $W(z)=\{(w_j(z),a)\}^v_{j=1}$ be a $v$-valued algebroid
function. The set of all algebroid mappings of $W(z)$ is denoted by
$Y_W$. The set $$H_W:=\{h\circ W(z); h\in Y_W \}$$ is called the
algebroid function class of $W(z)$.\\

Set $$X_W:=\{f \in H_W; T(r,f)=o[T(r,W)]~(r\rightarrow
\infty,~r\not\in E_f ) \},$$ where $E_f$ is a real number set of
finite linear measure depending on $f$. $X_W$ is called the small
algebroid function set of $W(z)$. The element in $X_W$ is called the
small algebroid function of $W(z)$.
\end{definition}
Note that the set $X_W$ contains all the finite or infinite complex
constants, all the small meromorphic functions and all the small
algebroid functions.

\begin{definition}
Let the set of all algebroid mappings of $W(z)$ be $Y_W$ and
$H_W:=\{h\circ W(z); h\in Y_W \}$. For any $h_1,h_2\in Y_W$, define\\
1)Addition: $(h_1+h_2)\circ W(z)=h_1\circ W(z)+h_2\circ W(z)$.\\
2)Subtraction: $(h_1-h_2)\circ W(z)=h_1\circ W(z)- h_2\circ W(z)$.\\
3)Multiplication: $(h_1\cdot h_2)\circ W(z)=(h_1\circ W(z))\cdot(h_2\circ W(z))$.\\
4)Division: $(\frac{h_1}{h_2})\circ W(z)=h_1\circ W(z)
\cdot\frac{1}{h_2}\circ W(z)$.\\
It is easy to prove that they satisfy Definition 2.4. Hence they are
all algebroid mappings. So $H_W$ is a linear space and is closed
with respect to Multiplication and Division.
\end{definition}

Suppose that $\{a_j(z)\}$,$\{b_i(z)\}$ are two group of analytic
functions defined in the complex plane $C$, without no common zeros.
The function $$ q[z,w]:=\frac{
a_n(z)w^n+a_{n-1}(z)w^{n-1}+...+a_0(z)}{b_m(z)w^m+b_{m-1}(z)w^{m-1}+...+b_0(z)}$$
is called rational complex function with meromorphic coefficients.
The set of all rational complex functions with meromorphic
coefficients is denoted by $Q[z,w]$. By the above definition,
Definitions 2.5 and 2.6, for any $q[z,w]\in Q[z,w]$, $q\circ
\{(w_{j}(z),a)\}= \{(q[z,w_{j}(z)],a)\}\in H_W$ is the algebroid
function. So $q[z,w]\in Y_W$.

 Especially, when $Q(z)$ is a single
valued rational function defined in the complex plane, $Q\circ
W(z):=\{Q\circ w_j(z),a\}$ is the $v$-valued algebroid function. If
$W(z)$ is irreducible and $Q$ is linear, then $Q\circ W(z)$ is
irreducible. If $q[z,w]=w\in Q[z,w]$, then $q\circ
\{(w_{j}(z),a)\}=\{( w_{j}(z),a)\}=W(z)$ is an identical mapping.

\begin{theorem}
Suppose that $h$ is an algebroid mapping of $v$-valued irreducible
algebroid function $W(z)=\{(w_{j}(z),a)\}$. If $h\circ W(z)$ is
reducible, then it can split to $n(\geq 1)$ number of $m$-valued
irreducible algebroid functions and $v=mn$.
\end{theorem}

\begin{proof}
By Theorem 2.3, we know that $h$ isn't injective.Namely there exist
two regular function element $(w_1(z),a)\ne (w_2(z),a)$, such that
the image elements $(h\circ w_1(z),a)=(h\circ w_2(z),a)$. By Theorem
2.2, $h\circ W(z)=\{(h\circ w_{j}(z),a)\}$ can split at least two
equal $m$-valued ($2m\leq v$) algebroid functions $$h\circ
W_1(z)=\{(h\circ w_1(z),a)\}=h\circ W_2(z)=\{(h\circ w_2(z),a)\}.$$

If $2m<v$, then there exist the regular function elements
$$(h\circ w_3(z),a)\in h\circ W(z)-h\circ W_1(z)-h\circ W_2(z)$$
and $(h\circ w_4(z),a)\in h\circ W_1(z)=\{(h\circ w_1(z),a)\}$ such
that $(h\circ w_3(z),a)=(h\circ w_4(z),a)$(Otherwise, since the
primary images $(w_3(z),a)$ and $(w_4(z),a)$ are connected, $(h\circ
w_3(z),a)$ and $(h\circ w_4(z),a)$ are also connected, which
contradicts the fact that $W_1(z)$ is an alhgebroid function.).
Hence by Theorem 2.1£¬from $(h\circ w_3(z),a)$ we can continue a
$m$-valued algebroid function $h\circ W_3(z)$ such that it equals to
$h\circ W_1(z)$. This work doesn't stop until we get $n$ same
$m$-valued algebroid functions with $nm=v$.
\end{proof}

\begin{corollary}
Suppose that $h$ is an algebroid mapping
 of $v$-valued irreducible algebroid function $W(z)=\{(w_{j}(z),a)\}$.
If $v$ is prime, then $h\circ W(z)$ is irreducible or $v$ same
meromorphic functions.
\end{corollary}

Dealing with the addition of two $v$-valued algebroid functions, we
get the following result.
\begin{theorem}
Let $W(z)=\{(w_t(z),a)\}$ and $M(z)=\{(m_t(z),a)\}\in H_W$ be two
$v$-valued algebroid functions. Then
$$T(r,W+ M)\leq T(r,W)+T(r,M)+\log 2.$$
$$T(r,W\cdot M)\leq T(r,W)+T(r,M).$$
\end{theorem}

\begin{proof}
Suppose that $W(z)$ and $M(z)$ are decomposed to $v$ simple-valued
branch $\{W_t(z)\}$ and $\{M_t(z)\}$ in the cutting complex plane.
Then
$$m(r,W+ M)=\frac{1}{v}\sum_{1\leq t \leq v}m(r,W_t(z)+ M_t(z))$$
$$=\frac{1}{v}\sum_{1\leq t\leq v}\frac{1}{2\pi}\int^{2\pi}_{0}
\log^+|W_t(re^{i\theta})+ M_t(re^{i\theta})| d \theta $$
$$ \leq \frac{1}{v}(v\log 2+\sum^v_{t=1}\frac{1}{2\pi} \int^{2\pi}_{0} \log^+|W_t(re^{i\theta})|d \theta
+ \sum^v_{t=1}\frac{1}{2\pi} \int^{2\pi}_{0}
\log^+|M_t(re^{i\theta})|d \theta) $$
$$=m(r,W(z))+ m(r,M(z))+\log 2.$$
         $$N(r,W+ M)=\frac{1}{v}\int^r_{0}\frac{n(t,W+M)-n(0,W+M)}{t} d t
+\frac{n(0,W+M)}{v}\ln r  $$
$$\leq \frac{1}{v} \int^r_{0}\frac{n(t,W)-n(0,W)}{t} d t+
\frac{n(0,W)}{v}\ln r+\frac{1}{v}\int^r_{0}\frac{n(t,M)-n(0,M)}{t}d
t+\frac{n(0,M)}{v}\ln r$$
$$=N(r,W)+N(r,M).$$
          $$m(r,W\cdot M)=\frac{1}{v}\sum_{1\leq t\leq v}\frac{1}{2\pi}\int^{2\pi}_{0}\log^+|W_t(re^{i\theta})
M_t(re^{i\theta})| d \theta $$
 $$ \leq \frac{1}{v}(\sum^v_{t=1}\frac{1}{2\pi} \int^{2\pi}_{0} \log^+|W_t(re^{i\theta})|d \theta +
\sum^v_{t=1}\frac{1}{2\pi} \int^{2\pi}_{0}
\log^+|M_t(re^{i\theta})|d \theta) $$
$$=m(r,W(z))+ m(r,M(z)).$$
   $$N(r,W\cdot M)=\frac{1}{v}\int^r_{0}\frac{n(t,W\cdot M)-n(0,W\cdot M)}{t} d t
+\frac{n(0,W\cdot M)}{v}\ln r$$
$$ \leq \frac{1}{v} \int^r_{0}\frac{n(t,W)-n(0,W)}{t} d t+ \frac{n(0,W)}{v}\ln r+
\frac{1}{v}\int^r_{0}\frac{n(t,M)-n(0,M)}{t}d t)+\frac{n(0,M)}{v}\ln
r$$
$$=N(r,W)+N(r,M).$$
Hence we get the conclusions of Theorem 2.5.
\end{proof}

\section{ Nevanlinna's second main theorem
concerning small algebroid functions}

Since in $H_W$, elements in $X_W$ can make addition, subtraction,
multiplication, division and differential, we have conditions to
investigate the theorem concerning small algebroid functions.
Referring to the method in [2, 6], we firstly obtain the Second Main
Theorem concerning small algebroid functions.

\begin{lemma}
Suppose that $W(z)=\{(w_t(z),a)\}$ is a $v$-valued nonconstant
algebroid functin in $\{|z|<R\}$, and
 $\{a_j(z)\}^p_{j=0}\subset X_W$ are $q$ distinct small algebroid function with respect to $W(z)$.
Then for any $r\in(0,R)$, we have
$$|m(r,\sum^q_{j=1}\frac{1}{W(z)-a_j(z)})-\sum^q_{j=1}m(r,\frac{1}{W(z)-a_j(z)})|=S(r,W),$$
where $$S(r,W)=O(\log(rT(r,f))),~(r\rightarrow \infty,~r\not\in E
),$$ $E$ is a positive real number set of finite linear measure.
\end{lemma}

\begin{proof}
By using the tree $Y$ through all branch points of $W(z)$, we cut
$W(z)$ into $v$ singule-valued branch $\{W_t(z)\}^v_{t=1}$.
Accordingly, we cut every $a_j(z)$ into $v$ singule-valued branch
$\{a_{j,t}(z)\}^v_{t=1}$. For any $t=1,2,...,v$, set $$
F_t(z):=\sum^q_{j=1}\frac{1}{W_t(z)-a_{j,t}(z)}\eqno(3.1)$$ and
$$m(r,F_t)\leq \sum^q_{j=1}m(r,\frac{1}{W_t(z)-a_{j,t}(z)})+\log  q.\eqno(3.2)$$

In order to obtain the lower bound of $m(r,F_t)$, for any $z$, set
$$\delta_t(z):= \min_{1\leq j<u\leq
q}\{|a_{j,t}(z)-a_{u,t}(z)|\}\geq 0.$$ Note that $\delta_t(z)$ is
the function of $z$, by the uniqueness theorem, its zeros must be
isolated. Take arbitrary $z\in \{z; \delta_t(z)\ne 0\}$.

Case 1. If for any $j\in \{1,2,...,q\}$, we have $$
|W_t(z)-a_{j,t}(z)|\geq \frac{\delta_t(z)}{2q},$$ then
$$\sum^q_{j=1}\log ^+\frac{1}{|W_t(z)-a_{j,t}(z)|}\leq
q\log ^+\frac{2q}{\delta_t(z)}.\eqno(3.3)$$

Case 2. If there exists some $u\in \{1,2,...,q\}$ such that $$
|W_t(z)-a_{u,t}(z)|\leq \frac{\delta_t(z)}{2q}.\eqno(3.4)$$ Then
when $j\ne u$, we have
$$|W_t(z)-a_{j,t}(z)|\geq
|a_{u,t}(z)-a_{j,t}(z)|-|W_t(z)-a_{u,t}(z)|\geq
\delta_t(z)-\frac{\delta_t(z)}{2q}=\frac{2q-1}{2q}\delta_t(z).$$
Hence by (3.4) we get
$$\frac{1}{|W_t(z)-a_{j,t}(z)|}\leq
\frac{1}{2q-1}\frac{2q}{\delta_t(z)}\eqno(3.5)$$
$$< \frac{1}{2q-1}\frac{1}{|W_t(z)-a_{u,t}(z)|}.\eqno(3.6)$$
By (3.1) and (3.6) we get
$$|F_t(z)|\geq \frac{1}{|W_t(z)-a_{u,t}(z)|}-\sum_{j\ne
u}\frac{1}{|W_t(z)-a_{j,t}(z)|}$$
$$\geq
\frac{1}{|W_t(z)-a_{u,t}(z)|}-\frac{q-1}{2q-1}\frac{1}{|W_t(z)-a_{u,t}(z)|}>\frac{1}{2|W_t(z)-a_{u,t}(z)|}.$$
Then by (3.5) we get
 $$\log ^+|F_t(z)|>\log
^+\frac{1}{|W_t(z)-a_{u,t}(z)|}-\log2$$
$$=
\sum^q_{j=1}\log ^+\frac{1}{|W_t(z)-a_{j,t}(z)|}-
 \sum_{j\ne u}\log ^+\frac{1}{|W_t(z)-a_{j,t}(z)|}-\log  2$$
$$\geq
\sum^q_{j=1}\log ^+\frac{1}{|W_t(z)-a_{j,t}(z)|}-\sum_{j\ne u}\log
^+\frac{2q}{(2q-1)\delta_t(z)}-\log  2$$ $$> \sum^q_{j=1}\log
^+\frac{1}{|W_t(z)-a_{j,t}(z)|}-q\log ^+\frac{2q}{\delta_t(z)}-\log
2.$$
 Combining (3.3), in two cases we have
$$\log ^+|F_t(z)|>\sum^q_{j= 1}\log ^+\frac{1}{|W_t(z)-a_{j,t}(z)|}
-q\log ^+\frac{2q}{\delta_t(z)}-\log  2.\eqno(3.7)$$ By definition,
for any $z\in \{z; \delta_t(z)\ne 0\}$,
 there exists $j(z)\ne u(z)$ such that $\delta_t(z)=a_{j(z),t}(z)-a_{u(z),t}(z)$.
 Hence we get
$$\frac{1}{\delta_t(z)}=\frac{1}{|a_{j(z),t}(z)-a_{u(z),t}(z)|}
\leq \sum_{1\leq j<u\leq q}\frac{1}{|a_{j,t}(z)-a_{u,t}(z)|}.$$
So
$$
\frac{1}{2\pi}\int^{2\pi}_{0}\ln^+\frac{d\theta}{\delta_t(re^{i\theta})}
\leq \sum_{1\leq j<u\leq q}\frac{1}{2\pi}\int^{2\pi}_{0}
\ln^+\frac{d\theta}{|a_{j,t}(re^{i\theta})-a_{u,t}(re^{i\theta})|}+O(1)$$
$$=\sum m(r,a_{j,t}-a_{u,t})+O(1)\leq \sum T(r,a_{j,t}-a_{u,t})+O(1)\leq $$
$$=\sum [T(r,a_{j,t})+T(r,a_{u,t})]+O(1)=S(r,W).\eqno(3.8)$$

Write $z=re^{i\theta}$, integrating (3.7) and combining (3.8), we
get
$$m(r,F_t)>\sum^q_{j= 1}m(r,\frac{1}{|W_t(z)-a_{j,t(z)}|})+S(r,W).$$
Then by (3.2), we get
$$|m(r,F_t)-\sum^q_{j= 1}m(r,\frac{1}{|W_t(z)-a_{j,t(z)}|})|<S(r,W).$$
So
$$|m(r,F_t(z))-\sum^q_{j=1}m(r,\frac{1}{W(z)-a_{j,t(z)}})|$$
$$=|\frac{1}{v}\sum^v_{t=1}m(r,F_t)-\sum^q_{j=1}
[\frac{1}{v}\sum^v_{t=1}m(r,\frac{1}{|W_t(z)-a_{j,t(z)}|})]|$$
$$\leq \frac{1}{v}\sum^v_{t=1}|m(r,F_t)-\sum^q_{j=1}m(r,\frac{1}{|W_t(z)-a_{j,t(z)}|})|<S(r,W).$$
\end{proof}

\begin{lemma}
Suppose that $W(z)$ is a $v$-valued nonconstant algebroid functin in
$\{|z|<R\}$ and $n$ is a positive integer. Then $\frac{W^{(n)}}{W}$
is the differential polynomial of $\frac{W^\prime}{W}$.
\end{lemma}

\begin{proof}
When $n=1$, the conclusion holds cleary.

Suppose that for $n=t$ we have
$$\frac{W^{(t)}}{W}=P(\frac{W^\prime}{W}),$$
where $P(\frac{W^\prime}{W})$ is the differential polynomial of
$\frac{W^\prime}{W}$. Since
$$(\frac{W^{(t)}}{W})^\prime= \frac{W^{(t+1)}}{W}-\frac{W^{(t)}}{W}\cdot \frac{W^\prime}{W},$$

$$ \frac{W^{(t+1)}}{W}=(\frac{W^{(t)}}{W})^\prime+\frac{W^{(t)}}{W}\cdot \frac{W^\prime}{W}$$
$$=[P(\frac{W^\prime}{W} )]^\prime+P(\frac{W^\prime}{W})\cdot \frac{W^\prime}{W}$$
is the differential polynomial of $\frac{W^\prime}{W}$.
\end{proof}

\begin{lemma}
Let $f_1,f_2,...,f_k,g\in H_W$. Then
$$ W(f_1,f_2,...,f_k):=
\left|\begin{array}{llll}
f_1&f_2&\cdots&f_k\\
f^\prime_1&f^\prime_2&\cdots&f^\prime_k\\
\cdots&\cdots&&\\
f^{(k-1)}_1&f^{(k-1)}_2&\cdots&f^{(k-1)}_k\\
\end{array}
\right| = g^kW(\frac{f_1}{g},\frac{f_2}{g},...,\frac{f_k}{g}).$$
\end{lemma}

\begin{proof}
(1) When $k=2$, we have
$$g^2W(\frac{f_1}{g},\frac{f_2}{g})=g^2\left|\begin{array}{ll}
\frac{f_1}{g}&\frac{f_2}{g}\\
\\
(\frac{f_1}{g})^\prime&(\frac{f_2}{g})^\prime\\
\end{array}
\right| =g^2\left|\begin{array}{ll}
\frac{f_1}{g}&\frac{f_2}{g}\\
\\
\frac{f^\prime_1g-f_1g^\prime}{g^2}&\frac{f^\prime_2g-f_2g^\prime}{g^2}\\
\end{array}
\right|$$
$$=g^2[
 \frac{f_1f^\prime_2g-f_1f_2g^\prime}{g^3}- \frac{f_2f^\prime_1g-f_2f_1g^\prime}{g^3}]=
f_1f^\prime_2-f_2f^\prime_1=W(f_1,f_2).$$

(2) Suppose that for positive integer $k$, we have
$$g^kW(\frac{f_1}{g},\frac{f_2}{g},...,\frac{f_k}{g})=W(f_1,f_2,...,f_k).$$
Then for $k+1$, we have
$$
g^{k+1}W(\frac{f_1}{g},\frac{f_2}{g},...,\frac{f_k}{g},\frac{f_{k+1}}{g})=
g^{k+1}\left|\begin{array}{lllll}
\frac{f_1}{g}&\frac{f_2}{g}&\cdots&\frac{f_k}{g}&\frac{f_{k+1}}{g}\\
(\frac{f_1}{g})^\prime&(\frac{f_2}{g})^\prime&\cdots&(\frac{f_k}{g})^\prime&(\frac{f_{k+1}}{g})^\prime\\
\cdots&\cdots&&\\
(\frac{f_1}{g})^{(k-1)}&(\frac{f_2}{g})^{(k-1)}&\cdots&(\frac{f_k}{g})^{(k-1)}&(\frac{f_{k+1}}{g})^{(k-1)}\\
(\frac{f_1}{g})^{(k)}&(\frac{f_2}{g})^{(k)}&\cdots&(\frac{f_k}{g})^{(k)}&(\frac{f_{k+1}}{g})^{(k)}\\
\end{array}
\right|$$
$$=g^{k+1}\sum^{k+1}_{n=1}(-1)^{k+1-n}(\frac{f_n}{g})^{(k)}
W(\frac{f_1}{g},...,\frac{f_{n-1}}{g},\frac{f_{n+1}}{g},...,\frac{f_{k+1}}{g})$$
$$=g\sum^{k+1}_{n=1}(-1)^{k+1-n}(\frac{f_n}{g})^{(k)}W(f_1,...,f_{n-1},f_{n+1},...,f_{k+1})$$
$$=g\sum^{k+1}_{n=1}(-1)^{k+1-n}
[\sum^k_{j=0}C^k_{j}f^{(j)}_n
(\frac{1}{g})^{(k-j)}]W(f_1,...,f_{n-1},f_{n+1},...,f_{k+1})$$

$$=g\sum^k_{j=0}C^k_{j}(\frac{1}{g})^{(k-j)}[\sum^{k+1}_{n=1}(-1)^{k+1-n}
f^{(j)}_n W(f_1,...,f_{n-1},f_{n+1},...,f_{k+1})]$$

$$=g\sum^k_{j=0}C^k_{j}(\frac{1}{g})^{(k-j)}
\left|\begin{array}{lllll}
f_1&f_2&\cdots&f_k&f_{k+1}\\
(f_1)^\prime&(f_2)^\prime&\cdots&(f_k)^\prime&(f_{k+1})^\prime\\
\cdots&\cdots&&\\
(f_1)^{(k-1)}&(f_2)^{(k-1)}&\cdots&(f_k)^{(k-1)}&(f_{k+1})^{(k-1)}\\
(f_1)^{(j)}&(f_2)^{(j)}&\cdots&(f_k)^{(j)}&(f_{k+1})^{(j)}\\
\end{array}
\right|$$
$$=gC^k_k(\frac{1}{g})^{(k-k)}
\left|\begin{array}{lllll}
f_1&f_2&\cdots&f_k&f_{k+1}\\
(f_1)^\prime&(f_2)^\prime&\cdots&(f_k)^\prime&(f_{k+1})^\prime\\
\cdots&\cdots&&\\
(f_1)^{(k-1)}&(f_2)^{(k-1)}&\cdots&(f_k)^{(k-1)}&(f_{k+1})^{(k-1)}\\
(f_1)^{(k)}&(f_2)^{(k)}&\cdots&(f_k)^{(k)}&(f_{k+1})^{(k)}\\
\end{array}
\right|$$
$$=W(f_1,...,f_{n-1},f_n,f_{n+1},...,f_{k+1}).$$
So the conclusion of Lemma 3.3 holds.
\end{proof}

\begin{lemma}
Suppose that  $A_q=\{a_j:=a_j(z)\}^q_{j=1}\subset X_W$ are $q\geq 1
$ distinct small algebroid fuctions. Let $L(s,A_q)$ denote the
vector space spanned by finitely many
$a^{p_1}_1a^{p_2}_2...a^{p_q}_q$, where integer $p_j\geq
0$($j=1,2,...,q$) and $\sum^q_{j=1}p_j=s(\geq 1)$. Let $\dim
L(s,A_q)$ denote the dimension of the vector space $L(s,A_q)$. Then
for any $\epsilon>0$, there exists $s\geq 1$ such that $$
 \frac{\dim L(s+1,A_q)}{\dim L(s,A_q)}<1+\epsilon.$$
\end{lemma}

\begin{proof}
Let $G(s,A_q)$ denote the set of
 the form $a^{p_1}_1a^{p_2}_2...a^{p_q}_q$,
and let $\#(s,A_q)$ denote the number of distinct element of
$G(s,A_q)$.

Using mathematical induction, we firstly prove that for any
$q>0,s>0$, we have
$$\#(s+1,A_q)= C^{s+1}_{q+s}.\eqno(3.9)$$

When $q=1$, for any integer $s\geq 1$,
$\#(s+1,A_1)=1=C^{s+1}_{1+s}$. (3.9) holds.

When $q=2$, for any integer $s\geq 1$,
$\#(s+1,A_2)=s+2=C^{s+1}_{2+s}$. (3.9) holds.

Suppose that for $q=k$ and any integer $s\geq 1$, we have
$\#(s+1,A_k)=C^{s+1}_{k+s}$.  Then for $q=k+1$, we have
 $$\#(s+1,A_{k+1})=\#(s+1,A_k)+\#(s,A_k)\cdot \#(1,A_1)+\#(s-1,A_k)\cdot \#(2,A_1)$$
$$+...+\#(1,A_k)\cdot \#(s,A_1)+\#(s+1,A_1)$$
$$=\#(s+1,A_k)+\#(s,A_k)+\#(s-1,A_k)+...+\#(2,A_k)+\#(1,A_k)+1$$
 $$=C^{s+1}_{k+s}+C^{s}_{k+s-1}+C^{s-1}_{k+s-2}+...+C^{2}_{k+1}+C^1_{k}+1
=1+\sum^s_{j=0}C^{j+1}_{k+j}.$$

Since $C^{j+1}_{k+j+1} = C^{j+1}_{k+j}+C^{j}_{k+j}$,  $
C^{j+1}_{k+j}=C^{j+1}_{k+j+1}- C^{j}_{k+j}$. Substituting it into
the above equality, we get
 $$\#(s+1,A_{k+1})=1+\sum^s_{j=0}(C^{j+1}_{k+j+1}- C^{j}_{k+j}).$$
 $$=1+(C^{s+1}_{k+s+1}- C^{s}_{k+s})
                          +(C^{s}_{k+s}- C^{s-1}_{k+s-1})
                          +(C^{s-1}_{k+s-1}- C^{s-2}_{k+s-2})
                          +(C^{s-2}_{k+s-2}- C^{s-3}_{k+s-3})+...$$
 $$+(C^{4}_{k+4}- C^{3}_{k+3})+(C^{3}_{k+3}- C^{2}_{k+2})
+(C^{2}_{k+2}- C^{1}_{k+1})+(C^{1}_{k+1}- C^{0}_{k})$$
 $$=1+C^{s+1}_{k+s+1}- C^{0}_{k}=C^{s+1}_{k+s+1}.$$

Then we prove that for any $q>0, s>0$, we have $$C^{s+1}_{q+s}\leq
q(q+1)s^q .\eqno(3.10)$$

When $q=1$, for any integer $s\geq 1$, $C^{s+1}_{1+s}=1\leq 2s$.
(3.10) holds.

When $q=2$, for any integer $s\geq 1$, $C^{s+1}_{2+s}=s+2\leq 6s^2
$. (3.10) holds.

Suppose that for $q=k$ and any integer $s\geq 1$, we have
$C^{s+1}_{k+s}\leq k(k+1)s^k $. Then for $q=k+1$, we get
$$C^{s+1}_{k+s+1}=C^{s+1}_{k+s}\frac{k+s+1}{k} \leq  k(k+1)s^k \frac{k+s+1}{k}$$
$$= (k+1)s^k (k+s+1)=(k+1)(k+2)s^{k+1}\frac{k+s+1}{ks+2s}\leq (k+1)(k+2)s^{k+1}.$$
This shows that (3.10) holds. Combining (3.9), for any $q>0,s>0$ we
have
$$\dim L(s+1,A_q)\leq \#(s+1,A_q)= C^{s+1}_{q+s}\leq q(q+1)s^q.\eqno(3.11)$$

Finally if Lemma 3.4 doesn't hold, then for any integer $s\geq 1$,
we have $$ \dim L(s+1,A_q)\geq (1+\epsilon)\dim L(s,A_q).$$ Hence
$$\dim L(s+1,A_q)\geq (1+\epsilon)\dim L(s,A_q)\geq...\geq
(1+\epsilon)^s \dim L(1,A_q) \geq (1+\epsilon)^s.$$
 Combining (3.11), we get
 $$ (1+\epsilon)^s\leq q(q+1)s^q.\eqno(3.12)$$
But $$\lim_{s\rightarrow\infty}\frac{(1+\epsilon)^s}{ s^q}=\infty.$$
This contradicts (3.12).
\end{proof}

\begin{theorem} (Nevanlinna's Second Main Theorem)

Suppose that $W(z)=\{(w_{j}(z),a)\}$ is a $v$-valued nonconstant
algebroid function in the complex plane $C$. $\{a_j\}^q_{j=1}\subset
X_W$ are $q\geq 2 $ distinct small algebroid functions of $W(z)$.
Then for any $\epsilon\in(0,1)$ and $r>0$, we have
$$m(r,W)+\sum^q_{j=1}m(r,\frac{1}{W(z)-a_j})=(2+\epsilon)T(r,W)+2N_x(r,W)+S(r,W).\eqno(3.13)$$
Its equivalent form is
 $$ (q-1-\epsilon)T(r,W)\leq N(r,W)+\sum^q_{j=1}
N(r,\frac{1}{W-a_j})+2N_x(r,W)+S(r,W)\eqno(3.14)$$
 or
$$(q-4v+3-\epsilon)T(r,W)\leq N(r,W)+\sum^q_{j=1}
N(r,\frac{1}{W-a_j})+S(r,W).\eqno(3.15)$$
\end{theorem}

\begin{proof}
Let $A_q=\{a_1,a_2,...,a_q\}$ and $L(s,A_q)$ denote the vector space
spanned by finitely many $a^{n_1}_1a^{n_2}_2...a^{n_q}_q$, where
$n_j\geq 0$($j=1,2,...,q$) and $\sum^q_{j=1}n_j=s$. For given $s$,
set $\dim L(s,A_q)=n$. Let $b_1,b_2,...,b_n$ denote a basis of
$L(s,A_q)$. Set $\dim L(s+1,A_q)=k$. Let $B_1,B_2,...,B_k$ denote a
basis of $L(s+1,A_q)$. By Lemma 3.4, for any $\epsilon>0$, there
exists some $s$ such that $$ 1\leq \frac{k}{n}<1+\epsilon.
\eqno(3.16)$$

Let $$ P(W):=W(B_1,B_2,...,B_k,Wb_1,Wb_2,...,Wb_n).$$ Since
$B_1,B_2,...,B_k,Wb_1,Wb_2,...,Wb_n$ are linearly independent,
$P(W)\not\equiv 0$. By the definition of the Wronskian determinant,
we get
$$ P(W)=\sum C_p(z)\prod^{n+k-1}_{j=0}(W^{(j)})^{p_j}
=W^n\sum
C_p(z)\prod^{n+k-1}_{j=0}(\frac{W^{(j)}}{W})^{p_j}.\eqno(3.17)$$
Since $m(r,W^\prime/W)=S(r,W)$, we get $$m(r,P(W)\leq
nm(r,W)+S(r,W).\eqno(3.18)$$

By Lemma 3.3, we get
$$ W(B_1,...,B_k,Wb_1,...,Wb_n)=P(W)
=W^{n+k}W(\frac{B_1}{W},...,\frac{B_k}{W},b_1,...,b_n).$$

(i) Suppose that $(q(z),z_0)$ is a meromorphic fuction element or
multivalent algebraic function element of $W(z)$. If $z_0$ is a
$\tau$-fold pole of $q(z)$, by the right of the above equality, it
can be see that outside the poles of the small algebroid functions
$\{B_i\}$,$\{b_j\}$, the order of pole of $P(W)$ at $(q(z),z_0)$ is
$(n+k)\tau$.

If $z_0$ is a zero of $q(z)$,
 by the left of the above equality, it
can be see that outside the poles of the small algebroid functions
$\{B_i\}$,$\{b_j\}$, $(q(z),z_0)$ isn't the pole of $P(W)$.

(ii) For any $1\leq t\leq k$, set
$$W_t(B_1,...,B_k,Wb_1,...,Wb_n):=W(B_1,...,B_{t-1},B_{t+1},...,B_k,Wb_1,...,Wb_n).$$
When $k< t\leq n+k$, set
$$W_t(B_1,...,B_k,Wb_1,...,Wb_n):=W(B_1,...,B_k,Wb_1,...,Wb_{t-1},Wb_{t+1},...,Wb_n).$$
Suppose that $(q(z),z_0)$ is any $\lambda$-sheeted algebraic
function element of $W(z)$ and $z_0$ isn't the pole of $q(z)$. Then
$z_0$ is at most the pole of $q^\prime(z)$ with the order
$\lambda-1$. By Lemma 3.3 we get
$$ P(W)=
    \sum^k_{t=1}[(-1)^{t+1}B_t\cdot W_t(B^\prime_1,...,B^\prime_k,(Wb_1)^\prime,...,(Wb_n)^\prime)]$$
$$+\sum^{k+n}_{t=k+1}[(-1)^{t+1}Wb_t\cdot  W_t(B^\prime_1,...,
B^\prime_k,(Wb_1)^\prime,...,(Wb_n)^\prime)]$$
$$=\sum^k_{t=1}[(-1)^{t+1}B_t\cdot  (W^\prime b_t+Wb^\prime_t)^{n+k-1}
 W_t(\frac{B^\prime_1}{(Wb_t)^\prime},...,\frac{B^\prime_k}{(Wb_t)^\prime},
\frac{(Wb_1)^\prime}{(Wb_t)^\prime},...,\frac{(Wb_k)^\prime}{(Wb_t)^\prime})$$
$$+\sum^{k+n}_{t=k+1}[(-1)^{t+1}Wb_t\cdot
  (W^\prime b_t+Wb^\prime_t)^{n+k-1}
 W_t(\frac{B^\prime_1}{(Wb_t)^\prime},...,\frac{B^\prime_k}{(Wb_t)^\prime}
 ,\frac{(Wb_1)^\prime}{(Wb_t)^\prime},...,\frac{(Wb_k)^\prime}{(Wb_t)^\prime}).$$
Hence outside the poles of the small algebroid functions
$\{B_i\}$,$\{b_j\}$, the order of pole of $P(W)$ at $(q(z),z_0)$ is
at most $(\lambda-1)(n+k-1)$.

Combining (i) and (ii), we get
         $$N(r,P(W))\leq (n+k)N(r,W)+(n+k-1)N_x(r,W)+  S(r,W).$$
By (3.18) we get
$$ T(r,P(W))\leq
nT(r,W)+kN(r,W)+(n+k-1)N_x(r,W)+S(r,W).\eqno(3.19)$$
 Suppose that $a$ is a linear combination of $\{a_j\}$, then
 $$P(W-a)=W(B_1,B_2,...,B_k,Wb_1-ab_1,Wb_2-ab_2,...,Wb_n-ab_n)$$
$$=W(B_1,B_2,...,B_k,Wb_1,Wb_2,...,Wb_n)\pm \sum W(B_1,B_2,...,B_k,...),$$
where the element "..." behind $B_k$ in $\sum
W(B_1,B_2,...,B_k,...)$ consists of $ab_j$. But $ab_j$ and
$B_1,B_2,...,B_k$ are linearly dependent, so we get $\sum
W(B_1,B_2,...,B_k,...)=0$. Hence we get $$ P(W-a)=P(W).\eqno(3.20)$$
By (3.17) and Lemma 3.2, we get
$$P(W)=W^n\cdot Q(\frac{W^\prime}{W}),\eqno(3.21)$$
where $Q(\frac{W^\prime}{W})$ is the differential polynomial of
$\frac{W^\prime}{W}$. Set
$$u_j:=W-a_j,\hspace{0.3cm}Q_j:=Q(\frac{u^\prime_j}{u_j}),\hspace{0.3cm}j=1,2,...,q.$$
By (3.20) and (3.21) we get $ P(W)=P(u_j)=u^n_jQ_j$, namely
$\frac{1}{(W-a_j)^n}=\frac{Q_j}{P(W)}$. Hence we get
$$\frac{1}{|W-a_j|}=\frac{|Q_j|^{\frac{1}{n}}}{|P(W)|^{\frac{1}{n}}}.\eqno(3.22)$$
Set
$$ F(z):=\sum^q_{j=1}\frac{1}{W(z)-a_j}.$$
By Lemma 3.1, we get
$$m(r,F)=m(r,\sum^q_{j=1}\frac{1}{W(z)-a_j})=
\sum^q_{j=1}m(r,\frac{1}{W(z)-a_j})+O_1(1).\eqno(3.23)$$
By (3.22)we
get
$$ |F(z)|\leq \sum^q_{j=1}\frac{1}{|W(z)-a_j|}\leq
\frac{1}{|P(W)|^{\frac{1}{n}}} \sum^q_{j=1}|Q_j|^{\frac{1}{n}}.$$
Then by (3.19) and (3.16), we get
$$ m(r,F) \leq
\frac{1}{n}m(r,\frac{1}{P(W)})+\frac{1}{n}\sum^q_{j=1}m(r,Q_j)+O(1)$$
$$\leq \frac{1}{n}T(r,P(W))-\frac{1}{n}N(r,\frac{1}{P(W)})+S(r,W)$$
$$\leq T(r,W)+\frac{k}{n}N(r,W)
+ \frac{n+k-1}{n}N_x(r,W) - \frac{1}{n}N(r,\frac{1}{P(W)})+S(r,W)$$
$$< T(r,W)+\frac{k}{n}N(r,W)
+ 2N_x(r,W) - \frac{1}{n}N(r,\frac{1}{P(W)})+S(r,W).\eqno(3.24)$$ By
(3.16),(3.23) and (3.24), we get
 $$m(r,W)+
\sum^q_{j=1}m(r,\frac{1}{W(z)-a_j})\leq
 \frac{k}{n}m(r,W)+m(r,F)$$
$$\leq(1+\frac{k}{n})T(r,W)+ 2N_x(r,W)+S(r,W)$$
$$<(2+\epsilon)T(r,W)+ 2N_x(r,W)+S(r,W).$$
Hence we get (3.13).\\
Note that $$m(r,\frac{1}{W(z)-a_j})\leq
T(r,W-a_j)-N(r,\frac{1}{W-a_j})+O(1)$$
$$=T(r,W)-N(r,\frac{1}{W-a_j})+S(r,W).\eqno(3.25)$$
Substituting (3.25) into (3.13), we get (3.14).
\end{proof}

\end{document}